\documentclass[11pt]{amsart}
\usepackage[utf8]{inputenc}
\usepackage{amssymb,amsmath,amsthm,enumerate,enumitem,colonequals,mlmodern,tikz-cd,microtype, stmaryrd}
\usepackage[cal=euler,bfcal,bb=px,bfbb]{mathalpha}
\usepackage[top=3.75cm, bottom=3cm, left=3.5cm, right=3.5cm]{geometry}

\makeatletter
\@namedef{subjclassname@2020}{\textup{2020} Mathematics Subject Classification}
\makeatother

\usepackage{xcolor}
\colorlet{darkblue}{blue!55!black}
\colorlet{darkcyan}{cyan!50!black}
\colorlet{darkgreen}{green!60!black}

\usepackage{hyperref}
\hypersetup{
    colorlinks=true,
    linkcolor= darkblue,
    urlcolor= darkcyan,
    citecolor= darkgreen,
}

\def\eqref#1{\textcolor{darkblue}{(\ref{#1})}}

\PassOptionsToPackage{hyphens}{url}
\usepackage{hyperref}
\hypersetup{bookmarksdepth=2}

\usepackage[nameinlink]{cleveref} 
\Crefformat{section}{#2\S#1#3} 
\Crefmultiformat{section}{#2\S\S#1#3}{ and~#2#1#3}{, #2#1#3}{, and~#2#1#3}
\crefname{hypothesis}{hypothesis}{hypotheses}
\Crefname{hypothesis}{Hypothesis}{Hypotheses}

\usepackage[pagewise]{lineno}
\let\oldequation\equation
\let\oldendequation\endequation
\renewenvironment{equation}{\linenomathNonumbers\oldequation}{\oldendequation\endlinenomath}
\expandafter\let\expandafter\oldequationstar\csname equation*\endcsname
\expandafter\let\expandafter\oldendequationstar\csname endequation*\endcsname
\renewenvironment{equation*}{\linenomathNonumbers\oldequationstar}{\oldendequationstar\endlinenomath}
\let\oldalign\align
\let\oldendalign\endalign
\renewenvironment{align}{\linenomathNonumbers\oldalign}{\oldendalign\endlinenomath}
\expandafter\let\expandafter\oldalignstar\csname align*\endcsname
\expandafter\let\expandafter\oldendalignstar\csname endalign*\endcsname
\renewenvironment{align*}{\linenomathNonumbers\oldalignstar}{\oldendalignstar\endlinenomath}


\newcounter{intro}
\newcounter{HypCounter}

\theoremstyle{plain}
\newtheorem{theorem}{Theorem}[section]
\newtheorem{lemma}[theorem]{Lemma}
\newtheorem{corollary}[theorem]{Corollary}

\theoremstyle{definition}

\newtheorem{definition}[theorem]{Definition}
\newtheorem{example}[theorem]{Example}

\newtheorem*{hypothesis*}{Hypothesis}
\newtheorem{notation}[theorem]{Notation}

\newtheorem{remark}[theorem]{Remark}

\newtheorem*{ack}{Acknowledgements}

\numberwithin{equation}{section}
\numberwithin{theorem}{section}

\title[Closedness of singular locus and generation]{Closedness of the singular locus \\ and generation for derived categories}

\author[S.~Dey]{Souvik Dey}
\address{S.~Dey,
Faculty of Mathematics and Physics,
Department of Algebra,
Charles University, 
Sokolovsk\'{a} 83, 186 75 Praha, 
Czech Republic, \url{https://orcid.org/0000-0001-8265-3301}}
\email{souvik.dey@matfyz.cuni.cz, dey0976@gmail.com}  

\author[P.~Lank]{Pat Lank}
\address{P.~Lank,
Department of Mathematics,
University of South Carolina, 
Columbia, SC 29208,
U.S.A.}
\email{plankmathematics@gmail.com}

\date{\today}



\begin{document}
    
\begin{abstract}
    This work is concerned with a relationship regarding the closedness of the singular locus of a Noetherian scheme and existence of classical generators in its category of coherent sheaves, associated bounded derived category, and singularity category. Particularly, we extend an observation initially made by Iyengar and Takahashi in the affine context to the global setting. Furthermore, we furnish an example a Noetherian scheme whose bounded derived category admits a classical generator, yet not every finite scheme over it exhibits the same property.
\end{abstract}

\maketitle

\section{Introduction}
\label{sec:introduction}

This note strengthens a relationship between the closedness of the singular locus of a Noetherian scheme $X$ and the existence of classical generators in the bounded derived category of coherent sheaves, denoted by $D^b_{\operatorname{coh}}(X)$. By doing so, it establishes a useful homological condition that can be used to probe whether an important topological property concerning singularities of the scheme is satisfied.

In a triangulated category $\mathcal{T}$, an object $G$ is called a \textit{classical generator} if the smallest thick subcategory of $\mathcal{T}$ containing $G$, denoted by $\langle G \rangle$, coincides with $\mathcal{T}$. Put differently, any object in $\mathcal{T}$ can be \textit{built} from $G$ utilizing a finite combination of shifts, cones, and retracts of finite coproducts. This concept was first introduced in \cite{BVdB:2003}.

There has been a very active front towards understanding conditions under which $D^b_{\operatorname{coh}} (X)$ admits a classical generator. These cases include quasi-excellent Noetherian schemes of finite Krull dimension \cite{Aoki:2021}, Noetherian schemes $J\textrm{-}2$ \cite{Elagin/Lunts/Schnurer:2020}, Noetherian schemes admitting a separator \cite{Jatoba:2021}, and several instances in the affine setting \cite{Dey/Lank/Takahashi:2023,Olander:2023,Iyengar/Takahashi:2016,Iyengar/Takahashi:2019}. 

An objective in this line of work is to explicitly identify classical generators in $D^b_{\operatorname{coh}}(X)$. The progress made here includes Noetherian schemes of prime characteristic \cite{BILMP:2023}, noncommutative techniques \cite{Bhaduri/Dey/Lank:2024}, and varieties over a field \cite{Rouquier:2008, Hanlon/Hicks/Lazarev:2023, Hanlon/Hicks:2023, Favero/Huang:2023, Pirozhkov:2023, Brown/Erman:2023, Lank/Olander:2024, Lank:2023}.

Let us motivate the objective behind our note. In $D^b_{\operatorname{coh}}(X)$, a perfect complex is an object which is locally quasi-isomorphic to a bounded complex of finite locally free sheaves. The triangulated subcategory of $D^b_{\operatorname{coh}}(X)$ consisting of perfect complexes admits a classical generator \cite{BVdB:2003}. 

A Noetherian scheme is regular if, and only if, every object in $D^b_{\operatorname{coh}}(X)$ is quasi-isomorphic to a perfect; see \cite{Neeman:2021} for recent developments. In other words, if there exists a perfect complex $P$ such that $\langle P \rangle = D^b_{\operatorname{coh}}(X)$, then $X$ is regular, and vice versa.

On the contrary, if $X$ fails to be regular, yet $D^b_{\operatorname{coh}}(X)$ still admits a classical generator, it follows that the regular locus of $X$ must be open; see \cite[Lemma 2.6]{Iyengar/Takahashi:2019} for the affine case. In this regard, there are varying singularities based on the openness of the regular locus.

The \textit{singular locus} of a Noetherian scheme $X$ comprises points $p$ where $\mathcal{O}_{X,p}$ is not a regular local ring, while its set-theoretic complement is called the \textit{regular locus}. We say $X$ is $J\textrm{-}0$ if its regular locus contains a nonempty open subset and $J\textrm{-}1$ if the regular locus is an open subset. For further details, refer to \Cref{rmk:classical_generator_singularity_category_j1}. This leads us to our main result.

\begin{theorem}\label{thm:singular_locus_and_generation_closed_subschemes}
    For a Noetherian scheme $X$, the following conditions are equivalent:
    \begin{enumerate}
        \item \label{thm:main1} Every closed integral subscheme is $J\textrm{-}0$.
        \item \label{thm:main2} Every closed integral subscheme is $J\textrm{-}1$.
        \item \label{thm:main3} $D^b_{\operatorname{coh}}(Z)$ admits a classical generator for every closed integral subscheme $Z$ of $X$
        \item \label{thm:main4} $D_{\operatorname{sg}}(Z)$ admits a classical generator for every closed integral subscheme $Z$ of $X$
        \item \label{thm:main5} $\operatorname{coh} Z$ admits a classical generator for every closed integral subscheme $Z$ of $X$.
    \end{enumerate}
    Moreover, if any of these conditions are satisfied, both $D^b_{\operatorname{coh}}(Y)$ and $D_{\operatorname{sg}}(Y)$ admit a classical generator for any closed subscheme $Y$ of $X$.
\end{theorem}

\Cref{thm:singular_locus_and_generation_closed_subschemes} connects the closedness of the singular locus and the existence of classical generators in various categories of interest that are constructed from coherent sheaves. For background on these categories, please see \Cref{sec:generation}. Our work is a globalization of \cite[Theorem 1.1]{Iyengar/Takahashi:2019}, which observed a similar result in the setting of an affine scheme. However, we give an independent proof without relying on their techniques, offering a geometrically flavored strategy. 

An important consequence to our result lies in its generality. We exhibit this with an example which satisfies the conditions of \Cref{thm:singular_locus_and_generation_closed_subschemes}, yet does not qualify as quasi-excellent nor $J\textrm{-}2$ scheme, ensuring the results of \cite{Aoki:2021} and \cite{Elagin/Lunts/Schnurer:2020} are not applicable. This example follows a construction by Nagata, cf. \Cref{ex:murayama}. As some other consequences, we also globalize \cite[Corollary 2.7, Proposition 2.8]{Iyengar/Takahashi:2019} and part of \cite[Theorem I, Proposition I]{Nagata:1959}; see \Cref{NagC,J2 crit}.  

\begin{ack}
    We thank Srikanth Iyengar and Josh Pollitz for comments on an earlier draft.
    Moreover, Lank would like to thank Takumi Murayama for
    references and correspondence in regards to \Cref{ex:murayama}. Dey was partially supported by the Charles
    University Research Center program No. UNCE/24/SCI/022 and a grant GA CR
    23-05148S from the Czech Science Foundation. Lank was partially supported by the National Science Foundation under Grant No. DMS-1928930 while the author was in residence at the Simons Laufer Mathematical Sciences Institute (formerly MSRI). Both authors are grateful for the anonymous referee for helpful suggestions to our work.
\end{ack}

\begin{notation}
    Let $X$ be a Noetherian scheme. We will consider the following triangulated categories:
    \begin{enumerate}
        \item $D(X)$ is the unbounded derived category of complexes of $\mathcal{O}_X$-modules
        \item $D_{\operatorname{Qcoh}}(X)$ is the unbounded derived category of complexes of $\mathcal{O}_X$-modules with quasi-coherent cohomology
        \item $D^b_{\operatorname{coh}}(X)$ is the derived category of bounded complexes of $\mathcal{O}_X$-modules with coherent cohomology
        \item $\operatorname{perf}X$ is category of perfect complexes on $X$, i.e. those objects in $D^b_{\operatorname{coh}}(X)$ which locally are quasi-isomorphic to a bounded complex of locally free sheaves of finite rank
        \item $D_{\operatorname{sg}}(X)$ is the Verdier quotient of $D^b_{\operatorname{coh}}(X)$ by $\operatorname{perf} X$, see \Cref{rmk:singularity_categories} for details.
    \end{enumerate}
\end{notation}

\section{Generation}
\label{sec:generation}

This section briefly discusses notions of generation for both triangulated and abelian categories. The primary sources of references are respectively \cite{BVdB:2003, Neeman:2021, Rouquier:2008} and \cite{Dao/Takahashi:2014, Dey/Lank/Takahashi:2023, Iyengar/Takahashi:2016}.

\subsection{Triangulated categories}
\label{sec:generation_triangulated}

Let $\mathcal{T}$ be a triangulated category with shift functor $[1]\colon \mathcal{T} \to \mathcal{T}$ and $\mathcal{S}$ be a subcategory of $\mathcal{T}$.

\begin{definition}\label{def:thick_subcategory_triangulated}
    \hfill
    \begin{enumerate}
        \item A full triangulated subcategory of $\mathcal{T}$ is \textbf{thick} if it is closed under direct summands. The smallest thick subcategory of $\mathcal{T}$ containing $\mathcal{S}$ is denoted $\langle \mathcal{S} \rangle$.
        \item Consider the following additive subcategories of $\mathcal{T}$:
        \begin{enumerate}
            \item $\operatorname{add}(\mathcal{S})$ is the strictly full subcategory of retracts of finite coproducts of shifts of objects in $\mathcal{S}$
            \item $\langle \mathcal{S} \rangle_0$ consists of all objects in $\mathcal{T}$ isomorphic to the zero object
            \item $\langle \mathcal{S} \rangle_1 :=\operatorname{add}(\mathcal{S})$
            \item $\langle \mathcal{S} \rangle_n := \operatorname{add}\{ \operatorname{cone}(\phi) : \phi \in \operatorname{Hom}_\mathcal{T} ( \langle \mathcal{S} \rangle_{n-1}, \langle \mathcal{S} \rangle_1) \}$ if $n\geq 2$.
        \end{enumerate}
    \end{enumerate}
\end{definition}

\begin{remark}
    In the notation of \Cref{def:thick_subcategory_triangulated}, there exists an exhaustive ascending chain of (additive) subcategories for the smallest thick subcategory containing $\mathcal{S}$. That is, $\langle \mathcal{S} \rangle_i$ is contained in $\langle \mathcal{S} \rangle_{i+1}$ for all $i$ and $\langle \mathcal{S} \rangle$ coincides with $\bigcup^{\infty}_{n=0} \langle \mathcal{S} \rangle_n$.
\end{remark}

\begin{definition}
    An object $G$ of $\mathcal{T}$ is called a \textbf{classical generator} if $\langle G \rangle = \mathcal{T}$. Additionally, if there exists $n\geq 0$ such that $\langle G \rangle_n = \mathcal{T}$, we say $G$ is a \textbf{strong generator}.
\end{definition}

\begin{example}
    \hfill
    \begin{enumerate}
        \item If $X$ is a regular Noetherian scheme, then $D^b_{\operatorname{coh}}(X)$ admits a classical generator, cf. \cite[Corollary 3.1.2]{BVdB:2003}. Additionally, if $X$ is separated and of finite Krull dimension, then there exists a strong generator for $D^b_{\operatorname{coh}}(X)$, cf. \cite[Theorem 0.5]{Neeman:2021}. In the special case $X$ is quasi-affine or a quasi-projective variety over a field, one can explicitly identify such objects, cf. respectively \cite[Corollary 5]{Olander:2023} and \cite[Theorem 4]{Orlov:2009}.
        \item Let $X$ be a Noetherian $F$-finite scheme (i.e. whose Frobenius morphism is finite). If $G$ is a classical generator for $\operatorname{perf}X$, then $F_\ast^e G$ is a classical generator for $D^b_{\operatorname{coh}}(X)$ whenever $e \gg 0$. Additionally, if $X$ is separated, then $F_\ast^e G$ is a strong generator. See \cite[Theorem A]{BILMP:2023} for details.
        \item $D^b_{\operatorname{coh}}(X)$ admits a strong generator for any quasi-excellent separated Noetherian scheme $X$ of finite Krull dimension, cf. \cite[Main Theorem]{Aoki:2021}.
        \item Suppose $\pi \colon \widetilde{X} \to X$ is an alteration of varieties over a field where $\widetilde{X}$ is a smooth projective variety. If $\mathcal{L}$ is an ample line bundle on $\widetilde{X}$, then $\mathbb{R}\pi_\ast (\bigoplus^{\dim \widetilde{X}}_{i=1} \mathcal{L}^{\otimes i})$ is a strong generator for $D^b_{\operatorname{coh}}(X)$. See \cite[Example 3.13]{Dey/Lank:2024} for details. If $X$ is a variety over a perfect field, then such alterations exist, cf. \cite[Theorem 4.1]{deJong:1996}.
    \end{enumerate}
\end{example}

\begin{remark}\label{rmk:classical_generator_singularity_category_j1}
    Let $X$ be a Noetherian scheme. For further background on the following, please refer to \cite[\href{https://stacks.math.columbia.edu/tag/07P6}{Tag 07P6}]{StacksProject} and \cite[\href{https://stacks.math.columbia.edu/tag/07R2}{Tag 07R2}]{StacksProject}.
    \begin{enumerate}
        \item The \textit{regular locus} of $X$, denoted $\operatorname{Reg}(X)$, is the collection of points $p$ of $X$ such that $\mathcal{O}_{X,p}$ is a regular local ring. The \textit{singular locus} of $X$ is the collection $\operatorname{Sing}(X):=X\setminus \operatorname{Reg}(X)$. 
        \item We say $X$ is $J\textrm{-}0$ if the regular locus of $X$ contains a nonempty open subset, and is $J\textrm{-}1$ if the regular locus of $X$ is an open subset. If $D_{\textrm{sg}}(X)$ admits a classical generator, then $X$ is $J\textrm{-}1$, cf. \cite[Lemma 2.6]{Iyengar/Takahashi:2019} for the affine case. 
        \item A Noetherian scheme $X$ is said to be $J\textrm{-}2$ if for every morphism $Y \to X$ which is locally of finite type the regular locus $\operatorname{Reg}(Y)$ is open in $Y$. If $X$ is a $J\textrm{-}2$ scheme, then $D^b_{\operatorname{coh}}(X)$ admits a classical generator, cf. \cite[Theorem 4.15]{Elagin/Lunts/Schnurer:2020}. 
        \item Any quasi-excellent Noetherian scheme is $J\textrm{-}2$, e.g. proper schemes over a complete local ring.
    \end{enumerate}
\end{remark}

\begin{remark}
    Let $X$ be a Noetherian scheme and $E$ an object of $D^b_{\operatorname{coh}}(X)$.
    \begin{enumerate}
        \item $\operatorname{Supp}(E):= \cup_{n=-\infty}^\infty \operatorname{Supp}(\mathcal{H}^n (E))$
        \item $E$ is \textit{supported} on a closed subscheme $Z$ of $X$ if $\operatorname{Supp}(E)$ is contained in $Z$
        \item $E$ is \textit{scheme-theoretically supported} on a closed subscheme $Z$ of $X$ if there is an object $E^\prime$ of $D^b_{\operatorname{coh}}(Z)$ such that $i_\ast E^\prime \cong E$ where $i$ is the associated closed immersion.
        \item Any object $E$ supported on a closed subscheme is scheme-theoretically supported on a nilpotent thickening of the closed subscheme, cf. \cite[Lemma 7.40]{Rouquier:2008}.
    \end{enumerate}
    Given a closed subscheme $Z$ of $X$, then we consider the following thick subcategories of $D^b_{\operatorname{coh}}(X)$:
        \begin{enumerate}
            \item $D^b_{\operatorname{coh,Z}}(X)$ is the strictly full subcategory of objects in $D^b_{\operatorname{coh}}(X)$ whose cohomology is supported in $Z$.
            \item $\operatorname{perf}_Z X$ is the strictly full subcategory of perfect complexes on $X$ whose cohomology is supported in $Z$.
        \end{enumerate}
\end{remark}

\begin{remark}\label{rmk:singularity_categories}
    Let $X$ be a Noetherian scheme. The \textbf{singularity category} of $X$, denoted $D_{\operatorname{sg}}(X)$, is the Verdier quotient of $D^b_{\operatorname{coh}}(X)$ by $\operatorname{perf}X$. Note that $X$ is regular if, and only if, $D_{\operatorname{sg}}(X)$ is trivial. Please see \cite{Orlov:2004,Buchweitz/Appendix:2021} for further background in both geometric and algebraic contexts. The thick subcategory of objects $E$ in $D_{\operatorname{sg}}(X)$ which are isomorphic to the image of an object $E$ in $D^b_{\operatorname{coh},Z} (X)$ under the quotient functor $D^b_{\operatorname{coh}}(X) \to D_{\operatorname{sg}}(X)$ is denoted $D_{\operatorname{sg,Z}}(X)$.
\end{remark}

\begin{remark}\label{rmk:scheme_verdier_localization_ELS}
    \hfill
    \begin{enumerate}
        \item Let $X$ be a Noetherian scheme. If $j\colon U \to X$ denotes an open immersion and $Z= X \setminus U$, then there exists a Verdier localization sequence:
        \begin{displaymath}
            D^b_{\operatorname{coh},Z}(X) \to D^b_{\operatorname{coh}}(X)\xrightarrow{\mathbb{L}j^\ast} D^b_{\operatorname{coh}}(U).
        \end{displaymath}
        \item If there exists a Verdier localization sequence:
        \begin{displaymath}
            \mathcal{K} \to \mathcal{T} \to \mathcal{T}/\mathcal{K}
        \end{displaymath}
        where $\mathcal{K}$ and $\mathcal{T}/\mathcal{K}$ admit classical generators, then so does $\mathcal{T}$, cf. \cite[Proposition 4.3]{Elagin/Lunts/Schnurer:2020}.
    \end{enumerate}
\end{remark}

\subsection{Abelian categories}
\label{sec:generation_abelian}

Let $\mathcal{A}$ be an abelian category and $\mathcal{S}$ be a subcategory of $\mathcal{A}$. 

\begin{definition}
    \hfill
    \begin{enumerate}
        \item $\operatorname{add}(\mathcal{S})$ is the smallest strictly full subcategory of $\mathcal{A}$ closed under direct summands of finite coproducts of objects of $\mathcal{S}$
        \item A strictly full additive subcategory $\mathcal{T}$ of $\mathcal{A}$ is said to be \textit{thick} if it is closed under direct summands and if given any object appearing in a short exact sequence
        \begin{displaymath}
            0 \to A \to B \to C \to 0
        \end{displaymath} 
        where the other two objects belong to $\mathcal{T}$, so does the third.
        \item The smallest thick subcategory of $\mathcal{A}$ containing $\mathcal{S}$ is denoted $\operatorname{thick}(\mathcal{S})$.
        \item An object $G\in \mathcal{A}$ is called a \textit{classical generator} if $\operatorname{thick}(G)=\mathcal{A}$.
        \item Consider the following additive subcategories:
        \begin{enumerate}
            \item $\operatorname{th}^0 (\mathcal{S})$ consists of all objects in $\mathcal{A}$ isomorphic to the zero object
            \item $\operatorname{th}^1 (\mathcal{S}):= \operatorname{add}(\mathcal{S})$
            \item $\operatorname{th}^n (\mathcal{S})$ is defined as the strictly full subcategory consisting of direct summands of objects which appear in a short exact sequence
            \begin{displaymath}
                0 \to A \to B \to C \to 0
            \end{displaymath}
            where the other two objects belong to $\operatorname{th}^{n-1}(\mathcal{S})$.
        \end{enumerate}
    \end{enumerate}
\end{definition}

The following is an easy generalization of \cite[Lemma 3.15]{Dey/Lank/Takahashi:2023}.

\begin{lemma}\label{lem:thick_abelian_filtration}
    For any subcategory $\mathcal{S}$ of $\mathcal{A}$, we have $\operatorname{thick}(\mathcal{S})= \bigcup^\infty_{i=0} \operatorname{th}^i (\mathcal{S})$.
\end{lemma}

\begin{proof}
    First, we show by induction that $\operatorname{th}^n (\mathcal{S})$ is contained in $\operatorname{thick}(\mathcal{S})$ for each $n$. Let $E$ belong to $\operatorname{th}^1 (\mathcal{S})$. Then $E$ is a direct summand of a finite direct sum of objects in $\mathcal{S}$, but any such object must belong to $\operatorname{thick}(\mathcal{S})$. Assume there exists $n$ such that $\operatorname{th}^k (\mathcal{S})$ is contained in $\operatorname{thick}(\mathcal{S})$ for each $1 \leq k \leq n$. Let $E$ be in $\operatorname{th}^{n+1}(\mathcal{S})$. There exists a short exact sequence in $\mathcal{A}$:
    \begin{displaymath}
        0 \to A\to B \to C \to 0
    \end{displaymath}
    where $E$ is a direct summand of one of the objects and the other two are objects of $\operatorname{th}^n (\mathcal{S})$. We have three cases to check, and they can all be shown similarly to one another, so assume that $A,B$ both belong to $\operatorname{th}^n (\mathcal{S})$. However, the definition of a thick subcategory tells us $C$ belongs to $\operatorname{thick}(\mathcal{S})$ via the two-out-of-three property for short exact sequences. Hence, we have shown that $\cup^\infty_{n=1} \operatorname{th}^n (\mathcal{S})$ is contained in $\operatorname{thick}(\mathcal{S})$.
    
    Lastly, we show $\bigcup^\infty_{i=0} \operatorname{th}^i (\mathcal{S})$ is a thick subcategory of $\mathcal{A}$. For each $N\geq n$, we know that $\operatorname{th}^n (\mathcal{S})$ is contained in $\operatorname{th}^N (\mathcal{S})$. Suppose there is a short exact sequence where two of the three objects belong to $\bigcup^\infty_{i=0} \operatorname{th}^i (\mathcal{S})$:
    \begin{displaymath}
        0 \to A \to B \to C \to 0.
    \end{displaymath}
    There are three cases to check, and as above, it suffices to check one of them. Assume that $A$ is an object in $\operatorname{th}^s (\mathcal{S})$ and $B$ is an object in $\operatorname{th}^t (\mathcal{A})$. Set $v = \max\{ s,t\}$. Then $A,B$ are objects of $\operatorname{th}^v (\mathcal{A})$, and so $C$ belongs to $\operatorname{th}^{v+1} (\mathcal{S})$. Moreover, the definition of each $\bigcup^\infty_{i=0} \operatorname{th}^i (\mathcal{S})$ ensures it is closed under direct summands, and so this shows the desired claim.
\end{proof}

\begin{lemma}\label{lem:abelian_thick_to_bounded_derived_thick}
    If $\operatorname{thick}(\mathcal{S})=\mathcal{A}$, then $\langle \mathcal{S} \rangle = D^b(\mathcal{A})$.
\end{lemma}

\begin{proof}
    We know from \Cref{lem:thick_abelian_filtration} that $\operatorname{thick}(\mathcal{S})= \bigcup^\infty_{i=0} \operatorname{th}^i (\mathcal{S})$. Choose an object $E$ in $D^b (\mathcal{A})$. There exists a distinguished triangle in $D^b (\mathcal{A})$:
    \begin{displaymath}
        Z(E) \to E \to B(E)[1] \to Z(E)[1]
    \end{displaymath}
    where $Z(E)$ is the complex of cycles and $B(E)$ is the complex of boundaries. Note that the differentials of both $Z(E),B(E)[1]$ is zero, so these are complexes of shifts of objects in $\mathcal{A}$. If one can show that $Z(E),B(E)$ belongs to $\langle \mathcal{S} \rangle$, then the desired claim holds as one would have $E$ is in $\langle \mathcal{S} \rangle$. It suffices to check that each object of $A$ belongs to $\langle \mathcal{S} \rangle$. However, an induction argument tells us $\operatorname{th}^n (\mathcal{S})$ is contained in $\langle \mathcal{S} \rangle_{2^{n-1}}$ for all $n$. The case $n=1$ is clear, so assume there exists $n$ such that $\operatorname{th}^k (\mathcal{S})$ is contained in $\langle \mathcal{S} \rangle_{2^{k-1}}$ for all $1\leq k \leq n$. Let $E$ belong to $\operatorname{th}^{n+1} (\mathcal{S})$. There exists a short exact sequence in $\mathcal{A}$:
    \begin{displaymath}
        0 \to A \to B \to C \to 0
    \end{displaymath}
    where $E$ is a direct summand of one of the objects and the other two are objects of $\operatorname{th}^n (\mathcal{A})$. This gives us a distinguished triangle in $D^b (\mathcal{A})$:
    \begin{displaymath}
        A \to B \to C \to A[1].
    \end{displaymath}
    There are three cases, lets prove it for the case where $A,B$ are in $\operatorname{th}^n (\mathcal{S})$ as the others follow similarly. The inductive hypothesis tells us $A,B$ belong to $\langle \mathcal{S} \rangle_{2^{n-1}}$. Hence, it follows that $B$ belongs to $\langle \mathcal{S} \rangle_{2^n}$. This completes the proof.
\end{proof}

\section{Results}
\label{sec:results}

This section establishes \Cref{thm:singular_locus_and_generation_closed_subschemes}. The following lemma might be known to experts, but we spell it out for the sake of convenience. For details on Serre subcategories, please refer to \cite[\href{https://stacks.math.columbia.edu/tag/02MN}{Tag 02MN}]{StacksProject}.

\begin{lemma}\label{lem:quotient_functor_reflects_isomorphism_up_to_generation}
    Let $\mathcal{A}$ be an abelian category, $\mathcal{B}$ be a Serre subcategory, and $\pi\colon \mathcal{A} \to \mathcal{A}/\mathcal{B}$ be the associated quotient functor. Let $\mathcal{D}$ be a thick subcategory of $\mathcal{A}$ containing $\mathcal{B}$. If there exists objects $D$ in $\mathcal{D}$ and $A$ in $\mathcal{A}$ such that $\pi(D) \cong \pi(A)$ in $\mathcal{A}/\mathcal{B}$, then $A$ is an object of $\mathcal{D}$.
\end{lemma}

\begin{proof}
    Let $f \colon \pi(D) \to \pi(A)$ be an isomorphism in $\mathcal{A}/\mathcal{B}$. There exists a short exact sequence in $\mathcal{A}/\mathcal{B}$:
    \begin{displaymath}
        0 \to \pi(D) \xrightarrow{f} \pi(A) \to 0 \to 0.
    \end{displaymath}
    The map $f$ is the image of a map $f^\prime \colon D^\prime \to A/A^\prime$ where $D^\prime, A^\prime$ are respectively subobjects of $D,A$ and $D/D^\prime, A^\prime$ belong to $\mathcal{B}$. We have a short exact sequence in $\mathcal{A}$:
    \begin{displaymath}
        0 \to D^\prime \to D \to D/D^\prime \to 0.
    \end{displaymath}
    If both $D$ and $D/D^\prime$ belong to $\mathcal{D}$ (i.e. $\mathcal{B}\subseteq \mathcal{D}$), then $D^\prime$ is in $\mathcal{D}$. Set $D_1 := D^\prime / \ker f^\prime$. Since $\pi(f^\prime)$ is a monomorphism, we know that $\ker f^\prime$ is in $\mathcal{B}$, and hence, $\ker f^\prime$ is in $\mathcal{D}$. There exists a short exact sequence in $\mathcal{A}$:
    \begin{displaymath}
        0 \to \ker f^\prime \to D^\prime \to D_1 \to 0.
    \end{displaymath}
    As both $\ker f^\prime$ and $D^\prime$ belong to $\mathcal{D}$, we have $D_1$ belongs to $\mathcal{D}$. Note that $D_1 \cong \operatorname{coim} f^\prime$ in $\mathcal{A}$. Let $f_1 \colon \operatorname{coim} f^\prime \to A/A^\prime$ be the natural map induced by $f^\prime$. There exists a short exact sequence in $\mathcal{A}$:
    \begin{displaymath}
        0 \to \operatorname{coim}f^\prime \to A/A^\prime \to \operatorname{coker} f_1 \to 0.
    \end{displaymath}
    Tying our work so far together, there exists a commutative diagram in $\mathcal{A}/\mathcal{B}$ (cf. \cite[$\S 3.1$ Corollaire 1]{Gabriel:1962}):
    \begin{displaymath}
        \begin{tikzcd}
            0 & {\pi(D)} & {\pi(A)} & 0 & 0 \\
            0 & {\pi(\operatorname{coim} f^\prime)} & {\pi(A/A^\prime)} & {\pi(\operatorname{coker} f_1)} & 0
            \arrow[from=1-1, to=1-2]
            \arrow[from=1-4, to=1-5]
            \arrow[from=2-1, to=2-2]
            \arrow["{\pi(f_1)}", from=2-2, to=2-3]
            \arrow[from=2-3, to=2-4]
            \arrow[from=2-4, to=2-5]
            \arrow["w"', from=1-4, to=2-4]
            \arrow["v"', from=1-3, to=2-3]
            \arrow[from=1-3, to=1-4]
            \arrow["f", from=1-2, to=1-3]
            \arrow["u"', from=1-2, to=2-2]
        \end{tikzcd}
    \end{displaymath}
    where $u,v,w$ are isomorphisms in $\mathcal{A}/\mathcal{B}$. This tells us that $\pi (f_1)$ is an isomorphism, and so both $\ker f_1$ and $\operatorname{coker} f_1$ belong to $\mathcal{D}$ as these belong to $\mathcal{B}$ (cf. \cite[$\S 3.1$ Lemme 4]{Gabriel:1962}). If both $\operatorname{coim} f^\prime$ and $\operatorname{coker} f_1$ are in $\mathcal{D}$, then $A/A^\prime$ belongs to $\mathcal{D}$. Once more, there is the short exact sequence in $\mathcal{A}$:
    \begin{displaymath}
        0 \to A^\prime \to A \to A/A^\prime \to 0.
    \end{displaymath}
    If $A^\prime$ and $A/A^\prime$ are in $\mathcal{D}$, then so is $A$, which shows the desired claim.
\end{proof}

\begin{lemma}\label{lem:thick_subcategory_correspondence_serre_quotients}
    Let $\mathcal{A}$ be an abelian category, $\mathcal{B}$ be a Serre subcategory, and $\pi\colon \mathcal{A} \to \mathcal{A}/\mathcal{B}$ be the associated quotient functor. There exists a bijective correspondence between thick subcategories $\mathcal{C}$ of $\mathcal{A}/\mathcal{B}$ and thick subcategories of $\mathcal{A}$ containing $\mathcal{B}$.
\end{lemma}

\begin{proof}
    Let $\mathcal{C}$ be a strictly full subcategory of $\mathcal{A}/\mathcal{B}$. Let $\pi \colon \mathcal{A} \to \mathcal{A}/\mathcal{B}$ be the quotient functor. Set $\pi^{-1}(\mathcal{C})$ to be the collection of objects $C$ in $\mathcal{A}$ such that $\pi(C)$ belongs to $\mathcal{C}$.
         
    Suppose $\mathcal{C}$ is thick. We show that $\pi^{-1}(\mathcal{C})$ is a thick subcategory of $\mathcal{A}$ containing $\mathcal{B}$. It is clear that $\mathcal{B}$ is contained in $\pi^{-1}(\mathcal{C})$ as $\mathcal{B}=\ker \pi$. Suppose we have a short exact sequence in $\mathcal{A}$ where two of the three objects belong to $\pi^{-1}(\mathcal{C})$. Applying $\pi$ gives short exact in $\mathcal{A}/\mathcal{B}$, and so the third object must belong to $\mathcal{C}$. Hence, the third object belongs to $\pi^{-1}(\mathcal{C})$. A similar argument tells us that $\pi^{-1}(\mathcal{C})$ is closed under direct summands. It is evident that $\pi(\pi^{-1}(\mathcal{C}))=\mathcal{C}$. Hence, we have shown that each thick subcategory of $\mathcal{A}/\mathcal{B}$ is the essential image of a thick subcategory of $\mathcal{A}$ containing $\mathcal{B}$. 
        
    It suffices to check that every thick subcategory $\mathcal{D}$ of $\mathcal{A}$ containing $\mathcal{B}$ is of the form $\pi^{-1}(\mathcal{D}^\prime)$ where $\mathcal{D}^\prime$ is a thick subcategory of $\mathcal{A}/\mathcal{B}$. Let $\mathcal{D}^\prime= \operatorname{thick}_{\mathcal{A}/\mathcal{B}}(\pi(\mathcal{D}))$. We show that $\mathcal{D} = \pi^{-1} (\mathcal{D}^\prime)$. If $\pi(\mathcal{D})$ is contained in $\mathcal{D}^\prime$, then $\mathcal{D}$ is contained in $\pi^{-1}(\mathcal{D}^\prime)$. It suffices to check the reverse inclusion. Let $E$ belong to $\pi^{-1}(\mathcal{D}^\prime)$. Then $\pi(E)$ is in $ \mathcal{D}^\prime$. By \Cref{lem:thick_abelian_filtration}, there exists an $n\geq 0$ such that $\pi(E)$ belongs to $\operatorname{th}^n (\pi(\mathcal{D}))$. We induct on $n$ to show that if $\pi (E)$ is in $\operatorname{th}^n (\pi(\mathcal{D}))$, then $E$ belongs to $\mathcal{D}$. If $n=0$, then $E\in\mathcal{B}\subseteq \mathcal{D}$. Suppose $\pi(E)$ is in $\operatorname{th}^1 (\pi(\mathcal{D}))$. There exists $E^\prime$ in $\mathcal{A}/\mathcal{B}$ such that $\pi(E)\oplus E^\prime$ is a finite coproduct of objects in $\pi(\mathcal{D})$. That is, $\pi(E) \oplus E^\prime \cong \oplus^N_{i=1}\pi(D_i)$ for $D_i$ in $\mathcal{D}$. As the quotient functor $\pi \colon \mathcal{A} \to \mathcal{A}/\mathcal{B}$ is exact (cf. \cite[$\S 3.1$ Proposition 1]{Gabriel:1962}), it follows that $\pi(E) \oplus E^\prime \cong \pi(\oplus^N_{i=1} D_i)$ in $\mathcal{A}/\mathcal{B}$. By \Cref{lem:quotient_functor_reflects_isomorphism_up_to_generation}, we see that $\oplus^N_{i=1} D_i$ is in $\mathcal{D}$. Choose $E^{\prime \prime}$ in $\mathcal{A}$ such that $\pi(E^{\prime \prime})\cong E^\prime$. Once more, it follows that $\pi(E\oplus E^{\prime \prime})\cong \pi (\oplus^N_{i=1} D_i)$. Hence, $E\oplus E^{\prime \prime}$ is in $\mathcal{D}$. As $\mathcal{D}$ is thick, it follows that $E$ is in $\mathcal{D}$. This shows the case $n=1$. 
    
    Assume there exists $n\geq 1$ such that for all $0\leq k \leq n$ and all objects $A$ of $\mathcal{A}$, if $\pi(A)$ is in $\operatorname{th}^k (\pi(\mathcal{D}))$, then $A$ belongs to $\mathcal{D}$. Suppose $E$ is an object of $\mathcal{A}$ such that $\pi(E)$ is in $\operatorname{th}^{n+1}(\pi(\mathcal{D}))$. There exists a short exact sequence 
    \begin{displaymath}
        0 \to M \to N \to L \to 0
    \end{displaymath}
    where two of the three objects belong to $\operatorname{th}^n (\pi(\mathcal{D}))$ and $\pi(E)$ is a direct summand of the third object. Without loss of generality, we can assume $\pi(E)$ is a direct summand of $L$. There exists a commutative diagram in $\mathcal{A}/\mathcal{B}$ (cf. \cite[$\S 3.1$ Corollaire 1]{Gabriel:1962}):
    \begin{displaymath}
        \begin{tikzcd}
            0 & {M} & {N} & {\pi(E)\oplus \pi(E^\prime)} & 0 \\
            0 & {\pi(M^\prime)} & {\pi (N^\prime)} & {\pi(L^\prime)} & 0
            \arrow[from=1-1, to=1-2]
            \arrow[from=1-4, to=1-5]
            \arrow[from=2-1, to=2-2]
            \arrow[from=2-2, to=2-3]
            \arrow[from=2-3, to=2-4]
            \arrow[from=2-4, to=2-5]
            \arrow["w"', from=1-4, to=2-4]
            \arrow["v"', from=1-3, to=2-3]
            \arrow[from=1-3, to=1-4]
            \arrow[from=1-2, to=1-3]
            \arrow["u"', from=1-2, to=2-2]
        \end{tikzcd}
    \end{displaymath}
    where $u,v,w$ are isomorphisms in $\mathcal{A}/\mathcal{B}$. The last row above is the direct image of the short exact sequence in $\mathcal{A}$:
    \begin{displaymath}
        0 \to M^\prime \to N^\prime \to L^\prime \to 0.
    \end{displaymath}
    Applying \Cref{lem:quotient_functor_reflects_isomorphism_up_to_generation}, we see that $M^\prime, N^\prime$ are in $\mathcal{D}$. Hence, $L^\prime$ is in $\mathcal{D}$. As $\pi(E\oplus E^\prime),\pi(L^\prime)$ are isomorphic in $\mathcal{A}/\mathcal{B}$, \Cref{lem:quotient_functor_reflects_isomorphism_up_to_generation} tells us once more that $E\oplus E^\prime$ belongs to $\mathcal{D}$. This completes the proof by induction.
\end{proof}

The following has a special case for the category of coherent sheaves on a Noetherian scheme, cf. \cite[Remark 4.5]{Elagin/Lunts/Schnurer:2020}.

\begin{lemma}\label{lem:serre_localization_sequence_generation}(Elagin-Lunts-Schn\"{u}rer)
    Let $\mathcal{A}$ be an abelian category, $\mathcal{S}$ be a Serre subcategory, and $\pi\colon \mathcal{A} \to \mathcal{A}/\mathcal{B}$ be the associated quotient functor. If there exists $B$ in $\mathcal{B}$ and $G$ in $ \mathcal{A}/\mathcal{B}$ such that $\operatorname{thick}(B)=\mathcal{B}$ and $\operatorname{thick}(G)=\mathcal{A}/\mathcal{B}$, then $\operatorname{thick}(B \oplus G) = \mathcal{A}$.
\end{lemma}

\begin{proof}
    By \Cref{lem:thick_subcategory_correspondence_serre_quotients}, the smallest thick subcategory generated by $G\oplus B$ coincides with $\mathcal{A}$.
\end{proof}

\begin{lemma}\label{lem:generation_coherent_subcategory_closed_subscheme_support}
    Let $X$ be a Noetherian scheme. If $i\colon Z \to X$ is a closed immersion and there exists $G$ in $\operatorname{coh}Z$ such that $\operatorname{thick}(G)=\operatorname{coh}Z$, then $\operatorname{thick}(i_\ast G)=\operatorname{coh}_Z X$.
\end{lemma}

\begin{proof}
    Note that $\operatorname{coh}_Z X$ is a thick subcategory of $\operatorname{coh}X$, and so $\operatorname{thick}(i_\ast G)$ is contained in $\operatorname{coh}_Z X$. We check the converse. Let $E$ be an object of $\operatorname{coh}_Z X$. Let $\mathcal{I}$ be the ideal sheaf corresponding to $Z$. There is an $n\geq 0$ such that $\mathcal{I}^n E=(0)$. This gives us a descending chain in $\operatorname{coh} X$:
    \begin{displaymath}
        (0)=\mathcal{I}^n E \subseteq \mathcal{I}^{n-1}E \subseteq \cdots \subseteq \mathcal{I}E \subseteq E.
    \end{displaymath}
    All quotients $\mathcal{I}^j E/\mathcal{I}^{j+1}E$ are annihilated by $\mathcal{I}$, and so these belong to $\operatorname{thick}(i_\ast G)$. Thus, $E$ belongs to $\operatorname{thick}(i_\ast G)$.
\end{proof}

\begin{lemma}\label{lem:reduce_to_integral}
    Let $X$ be a Noetherian scheme. Let $\pi_i \colon Z_i \to X$ denote the closed immersions from the irreducible components of $X$. If for each $1\leq i \leq n$ there exists $G_i$ in $\operatorname{coh}Z_i$ such that $\operatorname{thick}(G_i)=\operatorname{coh}Z_i$, then $\operatorname{thick}(\oplus^n_{i=1} \pi_{i,\ast} G_i)= \operatorname{coh}X$.
\end{lemma}

\begin{proof}
    Let $E$ be in $\operatorname{coh}X$. Use induction on number of irreducible components for $\operatorname{Supp}(E)$ to reduce to case where $E$ is supported on a closed integral subscheme, cf. \cite[\href{https://stacks.math.columbia.edu/tag/01YD}{Tag 01YD}]{StacksProject}. The desired claims follows from \Cref{lem:generation_coherent_subcategory_closed_subscheme_support}.
\end{proof}

\begin{proof}[Proof of \Cref{thm:singular_locus_and_generation_closed_subschemes}]
    The last claim follows from the first claim by \Cref{lem:reduce_to_integral}. So we only need to prove the first claim. This will be done by Noetherian induction on $X$. It is obvious for the empty scheme, so it may be assumed that $X$ is nonempty. We may further impose that $X$ is integral as it suffices to verify the first claim on each irreducible component of $X$ by \Cref{lem:reduce_to_integral}. Observe that we only need to prove the case where $Z=X$ because Noetherian induction tells us it is true for all properly contained closed integral subschemes of $X$.

    By \Cref{lem:abelian_thick_to_bounded_derived_thick}, \eqref{thm:main5} implies \eqref{thm:main3}. Moreover, as $X$ is integral, we see that \eqref{thm:main2} implies \eqref{thm:main1}. It is straightforward that \eqref{thm:main3} and \eqref{thm:main4} are equivalent by abstract nonsense on Verdier localizations. On one hand, if $D^b_{\operatorname{coh}}(X)$ admits a classical generator, then so must $D_{\operatorname{sg}}(X)$ as the natural functor $D^b_{\operatorname{coh}}(X) \to D_{\operatorname{sg}}(X)$ is a Verdier localization with kernel being $\operatorname{perf}X$. On the other hand, if $D_{\operatorname{sg}}(X)$ admits a classical generator, then so must $D^b_{\operatorname{coh}}(X)$; to see, apply (2) of \Cref{rmk:scheme_verdier_localization_ELS} and \cite[Corollary 3.1.2]{BVdB:2003}.

    Next we show \eqref{thm:main3} implies \eqref{thm:main2}. By \eqref{thm:main3}, $D_{\operatorname{sg}}(U)$ admits a classical generator for each open affine subscheme $U$ of $X$. Then \Cref{rmk:classical_generator_singularity_category_j1} implies $U$ must be $J\textrm{-}1$. So the claim follows as $X$ admits an affine open cover by $J\textrm{-}1$ schemes.

    Lastly we show \eqref{thm:main1} implies \eqref{thm:main5}. There exists a nonempty open affine $j \colon U \to X$ contained in the regular locus of $X$, and hence, a sequence of abelian categories (see \cite[Remark 4.5]{Elagin/Lunts/Schnurer:2020})
    \begin{displaymath}
        \operatorname{coh}_Z X \to \operatorname{coh}X \to \operatorname{coh}U.
    \end{displaymath}
    As $Z$ is properly contained in $X$, the induction hypothesis tells us there exists $G$ in $\operatorname{coh}Z$ such that $\operatorname{thick}(G)=\operatorname{coh}Z$. Hence, we know that $i_\ast G$ is a classical generator for $\operatorname{coh}_Z X$ via \Cref{lem:generation_coherent_subcategory_closed_subscheme_support}. The desired claim follows from \Cref{lem:serre_localization_sequence_generation}. This completes the proof.

\end{proof}

Now we present two direct consequences of \Cref{thm:singular_locus_and_generation_closed_subschemes}. The first one is a global version  of Nagata's regularity criteria; see \cite[Corollary 2.7]{Iyengar/Takahashi:2019}, \cite[32.A]{Matsumura} and \cite[Tag 07P9]{StacksProject}. 

\begin{corollary}\label{NagC}  Let $X$ be a Noetherian scheme. If every closed integral subscheme of $X$ is J-$0$, then $X$ is J-$1$.
\end{corollary}

\begin{proof} This is a direct consequence of \Cref{thm:singular_locus_and_generation_closed_subschemes} and \Cref{rmk:classical_generator_singularity_category_j1}(2).
\end{proof} 

The next consequence contains a global version of \cite[Proposition 2.8]{Iyengar/Takahashi:2019} and recovers part of \cite[Theorem 1, Proposition 1]{Nagata:1959}. 

\begin{corollary}\label{J2 crit} The following conditions are equivalent for a Noetherian scheme $X$. 

\begin{enumerate}
    \item $X$ is J-$2$.

    \item Every integral scheme $Y$, for which there exists a finite morphism $Y\to X$, is J-$0$. 

    \item For every finite morphism of schemes $Y\to X$, the  category $\operatorname{coh}(Y)$ admits a classical generator. 

    \item For every finite morphism of schemes $Y\to X$, the  category $D^b_{\operatorname{coh}}(Y)$ admits a classical generator. 

    \item For every integral scheme $Y$, for which there exists a finite morphism $Y\to X$, the category $D_{sg}(Y)$ admits a classical generator. 

    \item Every closed integral subscheme of $X$ is J-$2$. 
\end{enumerate}
    
\end{corollary}

\begin{proof} The proof of the equivalences of (1) through (5) is similar to the proof of \cite[Proposition 2.8]{Iyengar/Takahashi:2019} by remembering that closed immersions are finite morphism, compositions of finite morphisms are finite, and by using \Cref{thm:singular_locus_and_generation_closed_subschemes} and \Cref{rmk:classical_generator_singularity_category_j1}(2) respectively in place of \cite[Theorem 1.1]{Iyengar/Takahashi:2019} and \cite[Lemma 2.6]{Iyengar/Takahashi:2019}. Now for the equivalence of (1) through (5) with (6), we first see that (1)$\implies$ (6) is by definition. Now assume (6), and we want to prove $X$ is $J\textrm{-}2$. By (2), it is enough to prove that for every finite morphism $Y\to X$, $Y$ is $J\textrm{-}1$. So let $f:Y\to X$ be a finite morphism. By \Cref{NagC}, it is enough to prove that every closed integral subscheme of $Y$ is $J\textrm{-}0$. So let $i:Z\to Y$ be a closed integral subscheme of $Y$.  Then, there exists a closed integral subscheme $j : Z' \to X$ of $X$ and a morphism $g:Z \to Z'$ such that $f\circ i=j\circ g$. Since $f$ is a finite morphism, hence so is $j\circ g=f\circ i$. since $j$ is a closed immersion, so $g: Z\to Z'$ is a finite morphism. By hypothesis of (6), $Z'$ is $J\textrm{-}2$, hence $Z$ is $J\textrm{-}1$, so $J\textrm{-}0$ (as $Z$ is integral) as we wanted to show. 
\end{proof}

The following examples are new instances where \Cref{thm:singular_locus_and_generation_closed_subschemes} can be applied.

\begin{example}
    Any Nagata schemes of Krull dimension one is $J\textrm{-}2$. This is checked locally in light of \cite[\href{https://stacks.math.columbia.edu/tag/07PJ}{Tag 07PJ}]{StacksProject}. Appealing to \cite[Theorem 4.15]{Elagin/Lunts/Schnurer:2020} or \Cref{thm:singular_locus_and_generation_closed_subschemes}, it can be seen that $D^b_{\operatorname{coh}}(X)$ admits a classical generator.
\end{example}

\begin{example}
    Let $X=\operatorname{Spec}(R)$ where $R$ is a Noetherian semi-local integral domain of Krull dimension two with open regular locus. We can verify that $D^b_{\operatorname{coh}}(X)$ admits a classical generator. Since $X$ has finitely many closed points, it suffices to check the claim at each closed point. Indeed, if for each maximal ideal $p$ in $X$ one has $D^b_{\operatorname{coh}}(\mathcal{O}_{X,p})$ admitting a classical generator, then it is the image of an object in $D^b_{\operatorname{coh}}(X)$ under the Verdier localization functor $D^b_{\operatorname{coh}}(X) \to D^b_{\operatorname{coh}}(\mathcal{O}_{X,p})$. For each closed point $p$ in $X$, choose an object $G_p$ of $D^b_{\operatorname{coh}}(X)$ satisfying this condition. By \cite[Theorem 3.6]{Letz:2021}, we see that $G:= \oplus_{p\in \operatorname{mSpec}(R)} G_p$ is a classical generator for $D^b_{\operatorname{coh}}(X)$. 
    
    Consider the case where $R$ is local. If $Z$ is a properly contained closed subscheme of $X$, then $Z$ is the affine spectrum of Noetherian local ring of Krull dimension at most one. By \cite[\href{https://stacks.math.columbia.edu/tag/07PJ}{Tag 07PJ}]{StacksProject}, any such closed subscheme is $J\textrm{-}2$, and hence, $J\textrm{-}1$. Therefore, \Cref{thm:singular_locus_and_generation_closed_subschemes} tells us $D^b_{\operatorname{coh}}(X)$ admits a classical generator.
\end{example}

\begin{example}[Murayama]\label{ex:murayama}
    There is an example of a Noetherian scheme which satisfies condition $(2)$ of \Cref{thm:singular_locus_and_generation_closed_subschemes}, and yet is not $J\textrm{-}2$. That is, there exists a Noetherian scheme whose closed integral subschemes are $J\textrm{-}1$, but $X$ is not $J\textrm{-}2$. We follow the construction made by Nagata in \cite[$\S 4$]{Nagata:1959}. Let $k_0$ be a perfect field of characteristic $p\not=0$ and $v_1,\ldots,v_n, \ldots$ be infinitely many algebraically independent elements over $k_0$. Set $k$ to be the field $k_0(v_1,\ldots,v_n,\ldots)$. Choose analytically independent elements $x_1,x_2$ over $k$. Set $A=k^p \llbracket x_1,x_2\rrbracket [k]$. Then $A$ is a regular local ring. Let $p_1,\ldots,p_n,\ldots$ be infinitely many prime elements in $A$ such that $p_i A \not= p_j A$ if $i\not= j$. For $n\geq 0$, let $q_n = p_1 \cdots p_n$ and $c= \sum_i v_i q_i$. The singular locus of $\operatorname{Spec}(A[c])$ is not contained in any proper closed subset of $\operatorname{Spec}(A[c])$. Let $I=A[\frac{1}{x_1}]$ and $B:=A[c]$. Then $I$ is a Dedekind domain and $B[\frac{1}{x_1}]=I[c]$ is a finite $I$-algebra. However, the singular locus of $I[c]$ is not closed. Since $I$ is a Dedekind domain, it satisfies the property that every closed integral subscheme is $J\operatorname{-}1$.
\end{example}

\bibliographystyle{alpha}
\bibliography{mainbib}

\end{document}